\documentclass[12pt, reqno]{amsart}
\usepackage{amssymb}
\usepackage{amsfonts}
\usepackage{hyperref}

\date{}
\author{ Guy Bouchitt\'e}
\address{UFR Sciences, Universit\'e du Sud-Toulon-Var, 
BP20132, 83957 La Garde Cedex, France}
\email{bouchitte@univ-tln.fr}

\newtheorem{theorem}{{\bf Theorem}}[section]
\newtheorem{lemma}[theorem]{{\bf Lemma}}
\newtheorem{definition}[theorem]{{\bf Definition}}
\newtheorem{rmk}[theorem]{{\bf Remark}}

\newtheorem{proposition}[theorem]{{\bf Proposition}}

\newcommand{\dom}{{\rm Dom\hspace{.5pt}}}
\newcommand{\Ri}{{\rm I\kern -0.12em R}} 
\newcommand{\Ni}{{\rm I\kern -0.12em N}} 

\def\M{{\mathcal M}}
\def\med{\medskip\noindent}
\def\disp{\displaystyle}
\def\O{\Omega}
\def \argmin{\mathop{\rm argmin}}

\def \dps{\displaystyle}
\def \div {\mathop {\rm div}\nolimits}
\def \dom {\mathop {\rm dom}}
\def \e{\varepsilon}

\def \liminf {\mathop {\rm liminf} \limits}
\def \lim {\mathop {\rm lim} \limits}

 \def \max {\mathop {\rm max}}
\def \N {{{\rm I} \kern - .15 em {\rm N}}}

\def \P {{\mathcal P}}

\def \spt {\mathop {\rm spt}}

\def\trait (#1) (#2) (#3){\vrule width #1pt height #2pt depth #3pt}
\def\to{\rightarrow}
\def\wto{\rightharpoonup}

\def\e{\varepsilon}
\def\R{\mathbb{R}}
\def\N{\mathbb{N}}

\def\f{\varphi}

\def\e{\varepsilon}

\def\f{\varphi}
\def\to{\rightarrow}

\def\e{\varepsilon}
\def\R{\mathbb R}
\def\N{\mathbb N}

\def\bd{\partial}

\begin{document}

 \title{\bf Convex Analysis and Duality }
 \maketitle
 \begin{center}  ({\em  Encyclopedia of Mathematical physics, pp.642-652, 2006})
 \end{center}
 
 \medskip
Convexity is an important notion in non linear optimization theory as well as in infinite dimensional functional
analysis. As will be seen below, very simple and powerful tools will be derived from elementary duality
arguments  (which are byproducts of the Moreau-Fenchel transform and Hahn Banach Theorem). We will emphasize on
applications to a large range of variational problems. Some arguments of measure theory will be skipped.

\section{Basic convex analysis}

In the following, we denote by $X$ a normed vector space, $X^*$ the topological dual of $X$.
If a different topology from the normed topology is used on $X$, we will denote it by $\tau$. 
For every $x\in X$ and $A\subset X$, ${\mathcal V}_x$ denotes the open neighbourghoods of $x$ and ${\rm int}\, A , {\rm cl}\, A$  the
interior and the closure of $A$.  
We deal with extended real-valued functions
$f : X \to {\R} \cup \{+\infty\}$. We denote by ${\rm dom}\, f = f^{-1}({\R})$ and by ${\rm epi}\, f  =
\left\{(x,\alpha)\in X\times\R \ :\ f(x)\le\alpha \right\}$ the {\em domain} and
the {\em epigraph} of $f$ respectively. We say that 
$f$ is  {\em proper} if ${\rm dom}\, f\not = \emptyset$. Recall that
\ $f$ is {\em convex}  if for every $(x,y)\in X^2$ and $t\in [0,1]$, there holds
$$ f( tx+ (1-t)  y)\  \le\ t f(x) + (1-t) f(y)\quad \text{ (by convention $\infty + a =+\infty$)}\ .$$
The notion of convexity for a subset $A \subset X$ is recovered that by saying that  $\chi_A$ is convex, where
 its {\em indicator} function $\chi_A$ is defined by setting
$$ \chi_A (x)  =  0  \quad \text{ if  $x\in A$}\quad , \quad  \chi_A (x)  = +\infty  \qquad \text{ otherwise} \ .$$

\subsection{Continuity and lowersemicontinuity}

A first consequence of the convexity is the continuity on the topological interior of the domain.
We refer for instance to \cite{Le} for a proof of
\begin{theorem}\label{cont} Let $f : X \to {\R} \cup \{+\infty\}$ be convex and
proper. Assume that  $\disp\sup_U f
< +\infty$  where $U$ is a suitable  open subset  of $X$. Then $f$ is continuous and locally Lipchitzian on\
all ${\rm int}({\rm dom}\, f$).
\end{theorem}

 As an immediate corollary, a convex function on a normed space is continuous
provided it is majorized by a locally bounded function. In the finite dimensional case, it is easily deduced that
a finite valued convex function $f:\R^d \to \R$ is locally Lipschitz. Futhermore, by Aleksandrov's theorem, $f$ is almost everywhere 
twice differentiable and the nonnegative Hessian matrix $\nabla^2 f$ coincides with the absolutely continuous part of the distributional Hessian
matrix $D^2f$ (it is a Radon measure taking values in the non negative symmetric matrices).

 However in infinite dimensional spaces, for ensuring compactness properties
(as for instance in condition ii) of Theorem
\ref{exis} below), we
need to use weak topologies and the situation is not so simple. A major idea consists 
in substituting the continuity property with {\em lowersemicontinuity} (in short {\em lsc}).

\begin{definition}\  A function $f : X \to {\R} \cup \{+\infty\}$ is
$\tau$-lsc at $x_0 \in X$ if
for all $\alpha \in \R$,  there exists $U
\in {\mathcal{V}}_{x_0}$ such that $f>\alpha$\ on $U$.
In particular $f$ will be lsc on all $X$ provided $f^{-1}((r,+\infty))$ is open  for every  $r \in {\R}$  .
\end{definition}

\begin{rmk} \rm a) The following  sequential notion can be also used:  $f$ is {\em $\tau$-sequentially lsc} at $x_0$ if
$$
\forall (x_n) \subset X \quad  x_n \stackrel{{\tau}}{\to} x_0\quad \Longrightarrow
\quad\liminf_{n \to +\infty} f(x_n) \geq f(x_0)\ .
$$
It turns out that this notion (weaker in general) is  equivalent to the previous one provided $x_0$ admits a countable basis of
neighborhoods.

b) A well known consequence of Hahn Banach Theorem is that,  for convex functions, the lower semicontinuity  property with respect to
 the normed topology of $X$ is equivalent  to the weak (or weak sequential) lower semicontinuity.  
\end{rmk}

\begin{theorem} \label{exis} (existence)\  Let $f : X \to \R
 \cup \{+\infty\}$ be proper, such
that:\\
(i) $f$ is $\tau$-lsc\quad,\quad
(ii) $\forall r \in\R$, $f^{-1}((-\infty,r])$ is $\tau$-relatively
compact.

\med Then there is $\overline{x} \in X$ such that $f(\overline{x}) = \inf f$ and ${\rm
argmin}\, f :=\{x\in X\, |\, f(x) =
\inf f\}$ is $\tau$-compact.
\end{theorem}

In practice the choice of the topology $\tau$ is ruled by the condition ii) above. For example if $X$ is a reflexive infinite dimensional Banach space
and if $f$ is coercive (that is $\lim_{\Vert x\Vert \to \infty} f(x)=+\infty$), we may take for $\tau$ the weak topology (but never the normed topology). 
This restriction implies in practise that the first condition in Theorem \ref{exis} may  fail. In this case,
it is often useful to substitute
$f$ with its lower semicontinuous envelope.

\begin{definition}\label{relax} Given a topology $\tau$,  the {\em relaxed function} $\overline{f} \ (=\overline{f}^\tau)$ is defined as
$$
\overline{f} (x) \ =\ \sup \{g(x) | g :X \to \R \cup \{+\infty\}, g \mbox{ is }
{\tau}-lsc, g \leq f\}\ .
$$
\end{definition}
It is easy to check that
 $f$ is  $\tau$-lsc at $x_0$ if and only if\quad $\overline{ f }(x_0) =f(x_0)\ .$ Futhermore:
$$\overline{f}(x) = \sup_{U \in {\mathcal{V}}_u} \inf_U f\quad,\quad
{\rm epi}\, \overline{f}\, =\, {\rm cl}_{(X \times \R)}\, ({\rm epi} f)\ .$$
We can now state the relaxed version of Theorem \ref{exis}
\begin{theorem} (Relaxation)\  Let $f : X \to \R \cup \{+\infty\}$, then:\
 $\inf f = \inf \overline{f}\ .$ Assume futher that, for all real r, $f^{-1}((-\infty,r])$ is
 ${\mathcal{T}}$-relatively compact; then $f$ attains its minimum
and
${\rm Argmin}\,  f =
{\rm Argmin}\, \overline{f} \cap \{x \in X | f(x) =
\overline{f}(x)\}.  $
\end{theorem}

\subsection{Moreau-Fenchel conjugate}

The duality between $X$ and $X^*$ will be denoted by the symbol $\langle\cdot | \cdot \rangle$. If $X$ is an Euclidian space,
we identify $X^*$ to $X$ via the scalar product denoted $(\cdot|\cdot)$.  

\begin{definition} Let $f : X \to \R \cup \{+\infty\}$. The {\em Moreau-Fenchel conjugate} $f^* : X^* \to \R \cup \{+\infty\}$ of
$f$ is defined by setting, for every $x^* \in X^*$ :
$$ f^*(x^*) = \sup \{\langle x|x^* \rangle - f(x) | x \in X\}.$$
In a symmetric  way, if $f^*$ is proper on $X^*$, we define the {\em biconjugate} $f^{**}: X \to \R \cup\{+\infty\}$ by setting
$$  f^{**}(x) = \sup \{\langle x|x^* \rangle - f^*(x^*) | x^* \in X^*\}.$$

\end{definition}
As a consequence, the so called {\em Fenchel inequality} holds
$$ \langle x | x^* \rangle \ \le f(x)\ +\  f^*(x^*) \quad ,\quad  (x,x^*)\in X\times X^*\ .$$
Notice that $f$ does not need to be convex. However if $f$ is convex, then $f^*$ agrees with the Legendre-Fenchel transform.

\begin{definition} Let $f : X \to \R \cup \{+\infty\}$. The {\em subdifferential} of $f$ at $x$ is the possibly void subset of $\bd f(x)\subset X^*$ defined
by $$ \bd f(x) :=  \ \{ x^* \in X^* \ :\ f(x) + f^*(x^*) = \langle x, x^*\rangle\}\ .$$
\end{definition}
It is easy to check that $\bd f(x)$ is  convex and weak-star closed.
  Moreover, if $f$ is convex and has a differential (or Gateaux derivative) $f'(x)$ at
$x$, then $\bd f(x) = \{ f'(x)\}$.
After summarizing some elementary properties of the Fenchel transform, we give  examples in $\R^d$ or in infinite
dimensional spaces.
\begin{lemma} \label{fenchel}

\med (i)\ $f^*$ is convex,  lsc with respect to the weak star topology of $X^*$.

\med (ii)\ $f^*(0) = - \inf f\ $ and \ $f \geq g \Rightarrow f^* \leq g^*$.

\med (iii)\  $ \dps (\inf_i f_i)^* = \sup_i f_i^*\ ,$ for every family $\{f_i\}$.

\med (iv)\  $f^{**}(x) = \sup \left\{ g(x) \ :\ g\ \text{\rm affine continuous on $X$ and }\ g\le f\right\}\ $

(by convention, the supremum is identically $-\infty$ if no such $g$ exists ).

\end{lemma}

\begin{proof}  i) is  a direct consequence of the fact that $f^*$ can be written as the  supremum of 
functions $g_x$ where $g_x := \langle x | \cdot\rangle - f(x)$. Clearly these functions  are affine and weakly star continuous on $X^*$.
The  assertions ii), iii)  are trivial. To obtain iv), it is enough to observe that an affine function function $g$ of the form
$g(x) = \langle x, x^*\rangle - \beta$ satisfies $g\le f$ iff $ f^{*}(x^*)\le \beta$.

\end{proof}

\med \textbf{Example 1.} Let $f : X \to \R$, be defined by $f(x)=\frac{1}{p}\|x\|^p_X$, $1 < p < +\infty$, then:
$$f^*(x^*) = \frac{1}{p'} \Vert x^* \Vert^{p'}_{X^*} \quad,\quad \text{with}\ \frac{1}{p} + \frac{1}{p'} = 1\ ,$$
whereas, for $p=1$, we find $f^* =  \chi_{B^*}$ where $B^*=\{\Vert x^* \Vert \le 1\}$.

\med\textbf{Example 2.}\ Let $A\in \R^{d^2}_{\rm sym}$ be a  symmetric positive definite matrix and let
$\disp f(x):= \frac{1}{2} (Ax|x)\ (x\in\R^d)\ .$
Then, for all $y\in \R^d$, we have $\disp f^*(y)= \frac{1}{2} (A^{-1} y|y)\ .$
Notice that if $A$ has a negative eigenvalue, then $f^*\equiv +\infty$.

\med Particular examples on $R^d$ are also very popular. For instance:

\med
\textit{Minimal surfaces:}\   $f(x)= \sqrt{1+ |x|^2} \quad ,\quad f^*(y)=\begin{cases} - \sqrt{1- |y|^2} &  \textrm{if}\ |y|\le 1 \\
 +\infty& \textrm{otherwise}\end{cases} \ .$\\ 
\textit{Entropy:}\ $f(x)= \begin{cases} x\log x & \textrm{if}\  x\in \R_+ \\ +\infty &
\textrm{otherwise} \end{cases} \quad,\quad  f^*(y)= \exp\, (y\!-\!1)\ . $

\med
\textbf{Example 3.} Let $C\subset X$ be convex, and let $f = \chi_C$. Then:
$$f^*(x^*) = \sigma_C(c^*) = \sup_{x \in C} \langle x |
x^* \rangle \qquad ( \textit{support function of C})\ .$$
Notice that if $M$ is a subspace of $X$, then:\ $(\chi_M)^* = \chi_{M^\perp}$.
We specify now a  particular case of interest: 

Let $\O$ be a bounded open subset of $\R^n$. Take $X= C_0(\overline{\O}; \R^d)$ to be the Banach space of continuous functions on  
the compact $\overline{\O})$ with values in $\R^d$. As usual, we identify the dual $X^*$  with the space $\M_b(\overline{\O};\R^d)$ of $\R^d$-valued Borel
measures on $\overline{\O}$ with finite total variation. Let $K$ be a closed convex of $R^d$ such that $\O\in K$. 
Then $\rho^0_K(\xi) := \sup\{ (\xi|z) \ :\ z\in K\}\ $ 
is a nonnegative convex lsc and positively 1- homogeneous function on $\R^d$   (for example $\rho_K$ is the Euclidean norm if $K$ is the unit ball of
$\R^d$). Let us define $C:= \{ \f\in X\ :\ \f(x) \in K \ ,\ \forall x\in\O\}$. Then, we have:
\begin{equation}\label{H} (\chi_C)^* (\lambda) \ =\ \int_\O \rho^0_K(\lambda) \ :=\ 
\int_\O \rho^0_K \left( \frac{d\lambda}{d\theta}\right)\, \theta(dx)\ ,
\end{equation}
 where $\theta $ is any non negative Radon measure such that $\lambda \ll \theta$ (the choice of $\theta$ is indifferent).  
In the case where  $K$ is the unit ball, we recover the total variation of $\lambda$.

\med
\textbf{Example 4.} \textit{(Integral functionals)}\ Given $1 \leq p < +\infty$, $(\O,\mu, {\mathcal T})$ a  measured space and
$\varphi : \O \times \R^d \to [0,+\infty]\ ,$ a ${\mathcal T} \otimes B_{\Ri^d}$- measurable integrand.
Then the partial conjugate $\varphi^*(x,z^*) := \sup\{\langle z | z^* \rangle -\varphi(x,z)\ :\ z\in \R^d\}$
is a convex measurable integrand. Let us define:
$$
I_\f : u\in (L^p_\mu)^d    \to \int_\O \varphi(x, u(x)) d\mu\ \ \in \R\cup\{+\infty\}\ ,$$
and assume that $I_\f$ is proper. Then there holds $(I_\f)^*= I_{\f^*}$, where:
$$
(I_\f)^* : v\in  (L^{p'}_\mu)^d  \to \int_\O \varphi^*(x,v(x)) d\mu\ .$$ 

\section{Duality arguments}

\subsection{Two key results}

The first result related to the biconjugate $f^{**}$ is a consequence of the Hahn-Banach Theorem.
Recalling the assertion v) of Lemma \ref{fenchel}, we notice that the existence of an affine minorant for $f$ is equivalent to the
properness of $f^*$ ( that is $\exists x_0^*\in X^*\ :\ f^*(x_0^*)<+\infty$).

\begin{theorem} \label{bi} Let $f : X \to \R \cup \{+\infty\}$ be convex and
proper. Then

\med
(i) $f$ is lsc at $x_0$ if and only if  $f^*$ is proper and  $f^{**}(x_0) = f(x_0)$.\\
In particular, the lower semicontinuity of $f$  on all $X$ is equivalent to the identity $f\equiv f^{**}.$

\med (ii) If $f^*$ is proper, then $f^{**} = \overline{f}$.
\end{theorem}

\begin{proof} \ We notice that by Lemma \ref{fenchel}, $f^{**}\le f$ and $\ f^{**}$ is lower semicontinuous (even for the weak topology).
Therefore $f^{**}\le \overline{f}$ and moreover $f$ is lsc at $x_0$ if $ f^{**}(x_0)\ge f(x_0)$.
Conversely, if  $f$ is lsc at $x_0$, for every $\alpha_0 < f(x_0)$, there exits a neighbourghood $V$  of $x_0$ such that
$V \times (-\infty,\alpha_0] \cap \overline{{\rm epi} f} =\emptyset$. It follows that $\overline{{\rm epi} f}$ is a proper closed
 convex subset of $X\times\R$ which does not intersect the compact singleton $\{(x_0,\alpha_0)\}$. By applying Hahn-Banach strict separation Theorem,
there exists $(x_0^*,\beta_0) \in X^*\times \R\, $ such that:
$$  \langle x_0 , x_0^* \rangle + \alpha_0 \beta_0 \quad < \quad   \langle x , x_0^* \rangle + \alpha \beta_0\qquad \text{for all $(x,\alpha)\in {\rm epi}
f$} \ .$$
Taking $\alpha \to \infty$ and $x\in {\rm dom}f$, we find $\beta_0\ge 0$. In fact $\beta_0>0$ as the strict inequality above would be violated for $x=x_0$.
Eventually, we obtain that $f$ is minorized by the affine continuous function $g(x)= - \langle x -x_0, \frac{x_0^*}{\beta} \rangle + \alpha_0$.
Thus we conclude that $f^*$ is proper and that $f^{**}(x_0)\ge\alpha_0$.

\med
The assertion ii) is a direct consequence of the equivalence in i).

\end{proof}

\begin{theorem} \label{minstar}
Let $X$ be a normed space and let $f : X \rightarrow
[0,+\infty]$ be a convex and proper function, assume that $f$ is
continuous at 0, then

\med
i) $f^*$ achieves its minimum on $X^*$

\med
ii) $f(0) = f^{**}(0) = - \inf f^*$

\end{theorem}

\begin{proof} i)\   Let $M$ be an upperbound of $f$ on the ball $\{ \Vert x\Vert \le R\}$. Then
$$
f^*(x^*)\ \geq\ \sup \left\{\langle x, x^* \rangle -
f(x) \ :\ \Vert x \Vert \leq R \right\}\ \geq\  R \Vert x^* \Vert_{X^*} - M \ .
$$
Hence, for every $r$,  the set $\{x^* \in X^*:
\ f^*(x^*) \leq r\}$ is bounded, thus $\tau-$ relatively compact where $\tau$ is the weak-star topology on $X^*$.
By assertion i) of Lemma \ref{fenchel}, $f^*$ is $\tau$- lsc and Theorem \ref{exis} applies.\quad
 ii) \  By Theorem \ref{bi}, since $f$ is convex proper and lsc at $x_0=0$, we have
$f(0)=f^{**}(0) = - \inf f^*$.

\end{proof}

\subsection{Some useful consequences}

\begin{proposition} \label{conjugate_sum} \ (\textit{Conjugate of a sum})\quad
Let $f,g : X \to \R \cup \{+\infty\}$ be convex such that:
\begin{equation}\label{qualif} \exists x_0 \in X \ :\  \textrm{$f$ is continuous at $x_0$ and $g(x_0) < +\infty$}\ .
\end{equation}
Then: \\
(i) $(f+g)^*(x^*) = \inf\limits_{x^*_1+x_2^* = x^*}\{f^*(x^*_1) + g^*(x^*_2)\}$ (the equality holds in $\overline{\R}$).\\
(ii) If both sides of the equality in (i) are finite, then the infimum in the right-hand side is achieved.
\end{proposition}

\begin{proof}  Without any loss of generality, we may assume that $x^*=0$ ( we reduce to this case by substituting $g$ with $g - \langle\cdot, x^*\rangle$).
We let $$ h(p)\ =\ \inf \{f(x+p)+g(x)|x\in X\}\ .$$
 Noticing that
$(p,x)\mapsto f(x+p)+g(x)$ is convex, we infer that  $h(p)$ is convex as well. As $h$ is majorized by the function  $p\mapsto f(x_0+p)+g(x_0)$ which by
(\ref{qualif}) continuous at 0, we deduce from Theorem \ref{cont} and Theorem \ref{minstar} that $h(0)= h^{**}(0)$  and that $h^*$ achieves its infimum.
Now $h(0) = \inf (f+g) = -(f+g)^*(0)$ and
\begin{eqnarray*}
h^*(p^*) &=& \sup \{  \langle p,p^* \rangle - h(p)\ :\ p\in X\}\\
&=&\sup\{  \langle p,p^* \rangle -f(x+p)-g(x)\ :\ x\in X, p\in X\}\\
 &= &g^*(-p^*) + f^*(p^*).
\end{eqnarray*}
The assertions i) ii) follow since $ - h^{**}(0) = \min h^* = \min \left\{ g^*(-p^*) + f^*(p^*)\right\}.$
\end{proof}

\begin{proposition}\label{comp} \ (\textit{Composition})\quad Let X,Y two Banach spaces and
$A : X \mapsto Y$  be a linear operator with dense domain $D(A)$.
Let $\Psi : Y \to \R \cup \{+\infty\}$ be a convex, l.s.c. function and let $F\mapsto X$ be the convex functional 
defined by:
$$ F (u)=  \Psi ( A u) \quad \text{\rm if $u\in D(A)$}\quad,\quad F(u) = +\infty \quad \text{\rm otherwise}\ .$$
 Assume that there exists $u_0 \in D(A)$ such that $\Psi$ is continuous at $Au_0$.
 Then:

i) \ The Fenchel conjugate of $F$ is given by:
$$  \forall f\in X^*,\qquad F^*(f)=\inf \{ \Psi^*(\sigma)\ :\ \sigma \in Y^*,\  A^*\sigma=f\}\ ,$$
where, if both sides of the equality are finite, the infimum in the right-hand side is achieved.

ii) \ If in addition $Y$ is reflexive and  $\Psi$ is lsc coercice,  we have
\begin{equation}\label{relax}
\overline{F}(u) = F^{**}(u) = \inf \{\Psi(p) |\  (u,p) \in \overline{G(A)}
\}\ ,
\end{equation}
where $G(A)$ denotes the graph of $A$.
\end{proposition}

\begin{proof} \ i)\  Define  $H,K : X \times Y \to \R \cup \{+\infty\} $  by:
$$
H(u,p)=\chi_{{G(A)}}(u,p)\quad,\quad K(u,p)= \Psi(p).
$$
Then we have the identity $F^*(f)= (H+K)^*(f,0)$, where the  conjugate of $H+K$ is taken with repect to the duality $(X\times Y, X^*\times Y^*)$.
From the assumption, $K$ is continuous at $(u_0, Au_0)\in {\rm dom\, H}$. By Proposition~\ref{conjugate_sum}, we obtain
$$ (H+K)^* (f,0) = \inf_{(g,\sigma) \in X^*\times Y^*} \{K^*(f-g,\sigma) + H^*(g,-\sigma)\}\ .$$
After a simple computation, it is easy to check that:
\begin{equation*}
\begin{array}{ccllc}
H^*(g,-\sigma) &=& 0 \quad &\text{\rm if $A^* \sigma = f$}\quad  & ( +\infty\quad  \text{\rm otherwise})\ ,\\
K^*(f-g, \sigma) &=&  \Psi^*(\sigma) \quad  &\text{\rm if $ \ g = f$} \qquad   & (+\infty \quad\text{\rm otherwise})\ .
\end{array}
\end{equation*}

\med ii) Let $J(u):= \inf \{ \Psi(p) \ :\ (u,p) \in \overline{G(A)} \}.$
As observed for $F^*$ in the proof of i), we have the identity $J^*(f)= (H+K)^*(f,0)$. Therefore,  
in view of Theorem \ref{bi}, $\overline{F} = F^{**}= J^{**}$ and it is enough to prove that $J$ is convex l.s.c. proper.
Let us consider a sequence $(u_n)$ in $X$ converging to some $u\in X$. Without any loss of
generality, we may assume that  $\liminf J(u_n)=\lim J(u_n) < +\infty$. Then there is a
sequence $(p_n)$ such that, for every $n$, $(u_n, p_n)\in \overline{G}(A)$ and $J(u_n)\geq \psi (u_n)- 1/n$. As $\psi $ is coercive, $\{p_n\}$ is bounded
in the reflexive space $Y$ and possibly passing to a subsequence, we may assume that $p_n$ converges weakly to some $p$. Since $\overline{G(A)}$ is a
(weakly) closed subspace of $X\times Y$, we infer that $(u,p)$ as the limit of $(u_n,p_n)$ still
 belongs to $\overline{G(A)}$ . Thus we concude thanks to the (weak) lowersemicontinuity of $\Psi$:
$$ \liminf_{n} J(u_n) = \lim_n \Psi(p_n) \ge\ \Psi(p) \ \ge J(u) \ .$$

\end{proof}

An immediate consequence of Propositions \ref{conjugate_sum} and \ref{comp} is the following variant:
\begin{proposition}\label{combi}  Under the same notation as in Proposition \ref{comp}, let $\Phi:X \to \R \cup \{+\infty\}$ be a convex function
and assume that there exists $u_0 \in D(A)$ such that $F(u_0)<+\infty$ and $\Psi$ is continuous at $Au_0$.
 Then we have
$$ \inf_{u\in X}\left\{ \phi(u) + \Psi(Au) \right\}\ =\ \sup_{\sigma\in Y^*} \left\{- \phi^*(-A^*\sigma) - \Psi^*(\sigma) \right\}\ ,$$
where the supremum in the right hand side is achieved. Furthermore a pair $(\bar u, \bar \sigma)$ is optimal if only if it satisfies the 
relations: $\bar\sigma \in \bd  \Psi(A \bar u)\ $ and $\ -A^*\bar\sigma \in \bd \phi(\bar u)\ .$
\end{proposition}

\begin{rmk}\label{closable} \ \rm From the assertion ii) of Proposition \ref{comp}, we may conclude that $F$ is lsc whenever the operator $A$ is closed.
If  now $A$ is merely closable (with closure denoted by $\overline{A}$), we obtain
$$
\overline{F}(u) = G(\bar{A}u) \quad \text{\rm if $u\in\dom \overline{A}$}\quad,\quad \overline{F}(u) =+\infty \quad \text{\rm otherwise}\ .
$$
This is the typical situation when $F$ is an integral functional defined on smooth functions of the kind $F(u) =\int_\O f(x,\nabla u) \, dx\ ,$ where $\O$
is an bounded open subset of $\R^n$, $f:\O\times \R^n \to \R$ is a convex integrand with quadratic growth (i.e.
$ c |z|^2 \le f(x,z) \le C (1+|z|^2$ for suitables $C\ge c>0$). Then $X=L^2(\O), Y=L^2(\O;\R^n)$, $G(v)=\int_\O f(x,v(x)) \, dx$ and $A: u\in C^1(\O)
\mapsto\nabla u\in L^2(\O;\R^n)$. It turns out that $A$ is closable and that the domain of $\bar A$ characterizes the Sobolev space $W^{1,2}(\O)$ on which
$\bar A$ coincides with  the distributional gradient operator.

The situation is more involved if we consider  $F(u) =\int_\O f(x,\nabla u) \, d\mu\ ,\ $ being $\mu$ is a possibly concentrated Radon measure supported on
$\O$. In general the operator $A: u\in C^1(\O)\subset L^2_\mu(\O) \mapsto\nabla u\in L^2_\mu(\O;\R^n)$ is not closable and we need to come back to the
general formula (\ref{relax}). The general structure of $\overline{G(A)}$ has been given in \cite{BoBuSe,BoFr1,BoFr2} namely:
$$ (u,\xi) \in\overline{G(A)} \Longleftrightarrow u \in W^{1,2}_\mu \ , \ \exists \eta\in L^2_\mu(\O;\R^n)\ :\  \xi = \nabla_\mu u + \eta\ ,\
\eta(x) \in T_\mu(x)^\perp\ ,$$
where $T_\mu(x)), \nabla_\mu(x)$ are  suitable notions of tangent space and tangential gradient with respect to $\mu$, and $W^{1,2}_\mu$ denotes the domain
of the extended tangential gradient operator.
\end{rmk}

\begin{rmk} \ \rm The assertion ii) of Proposition\ref{comp} is not valid in the non reflexive case. 
In particular, for  $F(u) =\int_\O f(x,\nabla u) dx$ where $f(x,\cdot)$ has a linear growth at infinity, we need to take $Y$ as the space of
$\R^n$-values vector measures on  $\O$ and the  the relaxed functional $F^{**}$ needs to be indentified on the space $BV(\O)$ of integrable functions with
bounded variations. The computation of $F^{**}$ is a delicate problem for which we refer to \cite{BoDa,BoVa}. 
\end{rmk}

\begin{rmk} \ \rm  By duality
techniques, it is possible also to handle variational integrals of the kind $F(u) = \int_\O f(x,u(x), \nabla u(x))\, dx\ $
  even if the dependence of $f(x,u,z)$ with respect to $u$ is nonconvex. The idea consists in embedding the space $BV(\O)$ in the larger
space $BV(\O\times\R)$ through the map: $u \mapsto 1_u$ where $1_u$ is the characteristic function defined on $\O\times R$ by setting
$1_u(x,t) := 1$ if  $u(x)>t$ \ ,\  $1_u(x,t) := 0$ \  otherwise.  Then it is possible to show, under suitable conditions on the integrand $f$, that
there exists a convex, lsc, 1- homogeneous functional $G: BV(\O\times\R) \to \R \cup \{+\infty\}$ such that $\overline {F}(u)= G(1_u)$.
This functional $G$ is constructed as in the example 3 of section 1 taking $C$ to be a suitable convex subset of $C^0(\O\times\R)$. 
This nice new idea has been the key tool of the calibration method developed recently in \cite{BoAlDa}.

\end{rmk}

\section{Convex variational problems in duality}

\subsection{Finite dimensional case}\quad We sketch the duality scheme in two cases:

\subsubsection{Linear programming} Let $c \in \R^n$, $b\in \R^m$ and $A$ a $m\times n$ matrix. We denote 
by $A^T$ the transpose matrix.  We consider
the linear program:
$$ \inf \left\{ (c| x) \ :\ x\ge 0\ ,\ Ax \le b\right\}\leqno{(\mathcal P)}$$
and its perturbed version ($p\in \R^m$):
$$ h(p) := \inf \left\{ (c| x) \ :\ x\ge 0\ ,\ Ax + p \le b\right\}\ .$$
An easy computation gives:
\begin{equation}\label{dual00}\forall y\in \R^m\quad,\quad  h^*(y) \ =\ \begin{cases} - (b| y) & \text{\rm if}\quad A^T\, y +
c \le 0\ ,\ y\ge 0 \\ +\infty &
\text{\rm otherwise} \end{cases}
\end{equation}
\begin{lemma}\label{Farkas}\ Assume that $\ \inf ({\mathcal P})\,$ is finite. Then:\\
i)  $h$ is convex proper and lsc at $0$. \quad
ii) $ ({\mathcal P})$ has at least one solution. 
\end{lemma}
\begin{proof} \ We introduce the (n+m)$\times$(m+1) matrix $B$ defined by
$B:= \left( \begin{array}{cc} c^T & 0 \\ A & I_m \end{array}\right)$\ ( $I_m$ is the m dimensional identity matrix).
Denote $\{b_1,b_2, \dots, b_{n+m}\}\subset \R^{m+1}$ the columms of $B$ and $K$ the convex cone
$K:= \{\sum_{j=1}^{j=n+m} \lambda_j b_j\ :\ \lambda_j \ge 0\}$. By Farkas lemma this cone $K$ is closed. 

\med i)\ Let $\alpha:=\liminf\{ h(p) \ :\ p\to 0\}$. We have to prove that $\ \alpha \ge h(0)= \inf {\mathcal P}$. Let $\{p_\e\}$ be a sequence in $\R^m$
such that $p_\e\to 0$ and $\ h(p_\e)\to \alpha$. By the definition of $h$, we may choose $x_\e\ge 0$ such that $A x_\e \le b$ and $(c| x_\e) \to \alpha$.
Then we see that the column vector $\tilde {x_\e}$  associated with $(x_\e, b\!-\!A x_\e)\in \R^{n+m}$ satisfies:
 $B\, \tilde {x_\e}\in K$ and $ B\, \tilde {x_\e} \to  \left(\begin{array}{c} \alpha \\ b \end{array}\right)$. Therefore
$\left(\begin{array}{c} \alpha \\ b \end{array}\right)\in K$ and there exists ${\tilde x}=(x,x')$ such that
 $x\ge 0\ ,\ x'\ge  0\ ,\ (c| x)=\alpha$ and $ Ax+x'=b$. It follows that $x$ is admissible for $({\mathcal P})$ and
 then $(c| x)=\alpha \ge h(0)$.

\med ii)\ We repeat the proof of i) choosing $p_\e=0$ so that $\alpha= \inf ({\mathcal P})$.
\end{proof}

\med
Thanks to the assertion i) in Lemma \ref{Farkas}, we deduce from Theorem \ref{bi} that: $\inf ({\mathcal P})=h(0)=h^{**}(0) = \sup - h^*$.
Recalling (\ref{dual00}), we therefore consider the dual problem:
$$ \sup \left\{ - b\cdot y \ :\ y\ge 0\ ,\ A^T + c \ge 0\right\}\leqno({\mathcal P}^*)$$
\begin{theorem}\label{PL} \ The following assertions are equivalent:\
i) $({\mathcal P})$ has a solution \\
 ii) $({\mathcal P}^*)$ has a solution \\
 iii) There exists  $(x_0,y_0) \in \R^n_+\times \R^m_+$ such that:
$ A x_0 \ \le b \ ,\ A^T y_0 +c \ge 0 \ .$

\medskip In this case, we have $\min ({\mathcal P})= \max  ({\mathcal P}^*)$ and an admissible pair $(\bar x,\bar y)$  is optimal
if and only if $ c\cdot \bar x = - b \cdot \bar y$ or equivalenty satisfies the complementarity relations:\
 $ (A \bar x -b)\cdot \bar y\ =\ (A^T \bar y +c)\cdot \bar x =0 \ .$
\end{theorem}

\subsubsection{ Convex programming.}\quad Let  $f, g_1,\dots, g_m: X \to \R$ \ be convex lsc functions and the optimization problem
$$ \inf \left\{ f(x) \ :\ g_j(x) \le 0 \ ,\ j=1,2\dots ,m \right\} .\leqno({\mathcal P})\ .$$
Here $X=\R^n$ or any Banach space.
As before, we introduce the value function 
$$ p\in \R^m \quad,\quad h(p) := \inf\left\{ f(x)\ :\ g_j(x) + p_j \le 0\, j\in{1,2,\dots ,m} \right\}\ ,$$
and compute its Fenchel conjugate: 
$$ \lambda\in \R^m \quad,\quad h^*(\lambda)= \ - \inf_{x\in X}\left\{L(x,\lambda) \right\}\quad \text{\rm if}\ \lambda\ge 0\quad (+\infty \ \text{\rm
otherwise})\ ,$$ where $L(x,\lambda) := f(x) + \sum \lambda_i g_i(x)$ is the so called {\em Lagrangian}.
We notice that $h$ is convex and that the equality $h(0)=h^{**}(0)$ is equivalent to the {\em zero duality gap relation}
$$ \disp \inf_x \sup_\lambda L(x,\lambda)\ = \ \sup_\lambda \inf_x L(x,\lambda)\ .$$
 This condition is fulfilled in particular if we make the
following qualification assumption (ensuring that $h$ is continuous at $0$ and Theorem \ref{minstar} applies):
\begin{equation}\label {slater} \exists x_0\in X\ :\ f \ \text{continuous at $x_0$ \ ,\ $g_j(x_0) <0, \ \forall j$}\ .
\end{equation}
\begin{theorem}\label{PC} \ Assume that (\ref{slater}) holds. Then $\bar x$ is optimal for $({\mathcal P})$ if and only 
if there exist Lagrangian multipliers
$\bar\lambda_1, \bar\lambda_2, \dots \bar\lambda_m$ in $\R_+$ such that:
$$ \bar x \in \argmin_X  (f + \sum_j  \bar\lambda_j g_j) \quad,\quad  \bar\lambda_j g_j(\bar x)= 0\ ,\ \forall j\ . $$
\end{theorem}

Notice that the existence of such a solution $\bar x$ is ensured if for example $X=\R^n$ and if, for some $k>0$,
the function $f + k \sum_j  g_j\ $ is coercive.
\subsection{Primal-dual formulations in mechanics}\quad We present here the example of elasticity which motivated the pioneering work by Moreau J.J.
on convex duality techniques. Further examples can be found in \cite{EkTe}.
An elastic body is placed in a bounded domain $\O \subset\R^n$ whose boundary $\Gamma$ consists in two disjoint parts $\Gamma= \Gamma_0\cup\ \Gamma_1$.
The unknown $u:\O\to \R^n$ (deformation)  satisfies a Dirichlet condition $u=0$ on $\Gamma_0$ where the body is clamped. The system is subjected to a
surface load $g\in L^2(\Gamma_1;\R^n)$ and by a volumic load $f\in L^2(\O;\R^n)$. The static equilibrium problem has the following variational formulation:
$$ \inf_{u=0 \ \textrm{on}\ \Gamma_0} \left\{ \int_\O j(x,e(u))\, dx - \int_\O f\cdot u \, dx - \int_{\Gamma_1} g\cdot u \, d{\mathcal H}^{n-1}\right\}\
,\leqno({\mathcal P})$$ 
where $e(u):=\frac{1}{2} (u_{i,j}+ u_{j,i})$ denotes the symmetric {\em strain tensor} and $j: (x,z) \in\O\times \R^{n^2}_{\rm sym} \to\R_+$
is a convex integrand  representing the local elastic behaviour of the material.  We assume a quadratic growth as in Remark \ref{closable} 
(in the  case of linear elasticity, an isotropic homogeneous material
is characterized by the quadratic form $j(x,z)= \frac{\lambda}{2} |tr(z)|^2 + \mu |z|^2$, being $\lambda,\mu$ the Lam\'e
constants).

We apply Proposition \ref{combi} with $X=W^{1,2}(\O;\R^n), Y=L^2(\O;\R^{n^2}_{\rm sym}),\, Au=e(u)$ and let
$$ \begin{array}{lll}\Phi(u)&=& - \int_\O f\cdot u \, dx - \int_{\Gamma_1} g\cdot u \, d{\mathcal H}^{n-1} \quad \textrm{if $u=0$ on $\Gamma_0$}
 \quad,\quad(+\infty \ \textrm{otherwise})\\
\Psi(v) &=&  \int_\O j(x,v))\, dx \ . \end{array}$$
After  some computations, we may write the supremum appearing in Proposition \ref{combi} as our dual problem
$$ \sup \left\{ - \int_\O j^*(x,\sigma)\, dx \ : \
\sigma\in L^2(\O;\R^{n^2}_{\rm sym})\ ,\ - \textrm{div} \sigma =f \ \textrm{on}\ \O \ ,\ \sigma \cdot n = g \ \textrm{on}\ \Gamma_1\right\}\ ,
\leqno({\mathcal P}^*)$$ 
where $j^*$ is the Moreau-Fenchel conjugate with respect to the second argument and $n(x)$ denotes the exterior unit normal on $\Gamma$. The matrix valued
map $\sigma$ is called the {\em stress tensor} and
$j^*$ the {\em stress potential}. Note that the boundary condition for $\sigma\, n$ have to be understood in the sense of traces.

\begin{theorem}\label{elas} \ The problems $({\mathcal P})$ and $({\mathcal P}^*)$ have solutions and we have the equality:\
$\inf ({\mathcal P})
=\sup({\mathcal P}^*)$. Futhermore, a pair $(\bar u,\bar\sigma)$ is optimal if and only if it satisfies the following system:
$$\left\{\begin{array}{cclll} -\div \bar \sigma &=& f & \textrm{on }\ \O &\quad (equilibrium) \\
\bar \sigma (x) &\in& \bd j(x, e(\bar u)) & \textrm{a.e. on}\ \O  & \quad \text{(constitutive law)}\\
u&=& 0 & \textrm{a.e. on}\ \Gamma_0 & \\
\sigma\, n &=& g  & \textrm{on }\ \Gamma_1 &
\end{array}\right.$$
\end{theorem}

\section{Duality in mass transport problems} 
\subsection{General cost functions}
Let $X,Y$ be a compact metric space and $c: X\times Y \to [0,+\infty)\ $ be a continuous cost function.
We denote by $\P(X), \P(X\times Y)$ the set of probability measures on $X$ and $X\times Y$ respectively.
Given two elements  $\mu\in \P(X), \nu \in\P(Y)$, we denote by $\Gamma(\mu,\nu)$ the subset of probablity measures in $\P(X\times Y)$
whose marginals are respectively $\mu$ and $\nu$. Identified as a subset of $(C^0(X\times Y))^* $ (the space of signed Radon measures on
$X\times Y$), it is convex and weakly-star compact.
The {\em Monge-Kantorovich} formulation of the {\em mass transport  problem} reads as follows:
\begin{equation}\label{MK} T_c(\mu,\nu)\ :=\ \inf \left\{  \int_{X\times Y} c(x,y) \, \gamma(dxdy) \ :\ \gamma\in \Gamma(\mu,\nu) \right\} \ .
\end{equation}
This formulation, where the infimum is achieved (as we minimize a lsc functional on a compact set  for the weak star topology),
is already a relaxation of the initial {\em Monge} mass transport problem
$$ \inf_T \left\{\int_X c(x, Tx) \, \mu(dx)\ :\ T^\sharp(\mu)=\nu  \right\}\ ,$$
where the infimum is searched among all transports maps $T:X \mapsto Y$ {\em pushing forward} $\mu$ on $\nu$ (i.e. such that $\mu ( T^{-1}(B)=\nu(B)$ for
all Borel subset $B\subset Y$). This is equivalent to  restrict the infimum in (\ref{MK}) to the subclass  $\{\gamma_T\} \subset\Gamma(\mu,\nu)$ where
$ \langle\gamma_T, \phi(x,y) \rangle := \int_X \phi(x,Tx) \mu(dx)\ .$
In order to find a dual problem for (\ref{MK}), we fix $\nu\in\P(Y)$ and consider the functional $F: \M_b(X) \to [0,+\infty)$ defined by
$$ F(\mu)= T_c(\mu,\nu) \quad \textrm{if $\mu \ge 0\ ,\ \mu(X)=1$} \quad, \quad  F(\mu)= +\infty \quad \textrm{otherwise} \ ,$$
($\M_b(X)$ denote the Banach space of (bounded) signed Radon measures on X).

\begin{lemma}\label{polar} \  $F$ is convex, weakly-star lsc and proper. Its Moreau-Fenchel conjugate is given by
$$ \forall \f \in C^0(X)\quad ,\quad  F^*(\f) = \ -\int_Y  \f^c(y) \ \nu(dy)\ ,$$
where
$$  \f^c(y)\ :=\ \inf\, \{ c(x,y)-\f(x)\, :\, x\in X\}\ .$$
\end{lemma} 

\begin{proof} \ The convexity property is obvious and the properness follows from the fact that
$F(\mu) \le \int_{X\times Y} c(x,y)\ \mu\otimes \nu(dxdy)$. Let $\mu_n$ such that $\mu_n \wto \mu$ (weakly star). We
may assume that
$\liminf_n F(\mu_n) = \lim_n F(\mu_n) := \alpha$ is finite. Then $\mu_n$  and the associated optimal $\gamma_n$ are probability measures on $X$ and on
$X\times Y$ respectively. As $X$ and $Y$ are compact, possibly passing to a subsequence, we may assume that $\gamma_n \wto \gamma$ and clearly we have
$\gamma\in \Gamma(\mu,\nu)$. Since $c(x,y)$ is lsc non negative, we conclude that:
$$ \liminf_n F(\mu_n)\ = \liminf_n \int_{X\times Y} c(x,y) \, \gamma_n(dxdy) \ \ge  \ \int_{X\times Y} c(x,y) \, \gamma(dxdy) \ =\ F(\mu)\ .$$
Let us compute now $F^*(\f)$. We have:
$$ \begin{array}{lcl} - F^*(\f) &=& \inf \left\{ \int_{X\times Y} c(x,y) \,\gamma(dxdy) \ -\ \int_X \f \, d\mu \quad :\quad \mu\in \P(X), \gamma\in
\Gamma(\mu,\nu)\right\}\\
&=&  \inf \left\{ \int_{X\times Y} (c(x,y)- \f(x)) \,\gamma(dxdy)\quad:\quad  \gamma\in
\Gamma(\mu,\nu)\right\} \\
&\ge& \int_Y \f^c(y) \, \nu(dy) \ .
\end{array}$$
To prove that the last inequality is actually an equality, we observe that, for every $y\in Y$ and $\f\in C^0(X)$, the minimum of the lsc function
$c(\cdot,y)-\f$ is attained on the compact set $X$ and there exists a Borel selection map $S(y)$ such that $\f^c(y)= c(S(y),y) -\f(S(y)$ for all $y\in Y$.
We obtain the desired equality by choosing $\gamma$ defined, for every test $\psi$, by $\int_{X\times Y} \psi(x,y)\, \gamma(dxdy) := \int_Y
\psi(S(y),y)\, \nu(dy)\ .$

\end{proof}

We observe that, for every $\f\in C^0(X)$,  the function $\f^c$ introduced in Lemma \ref{polar} is continuous (use the uniform continuity of $c$) and
therefore the pair $(\f,\f^c)$ belong to the class
$$ {\mathcal F}_c:= \left\{ (\f,\Psi) \in C^0(X)\times C^0(Y)\ :\  \f(x) + \psi(y) \le c(x,y)\right\}\ .$$ 
Let us introduce the dual problem of (\ref{MK}):
\begin{equation}\label{MK*} \sup \left\{  \int_X \f\, d\mu \ +\ \int_Y \psi\, d\nu\ :\ (\f,\psi) \in {\mathcal F}_c \right\} \ .
\end{equation}
We will say that  $(\f,\psi)\in {\mathcal F}_c$ is a pair of {\em c-concave conjugate} functions if $\psi= \f^c$ and $\psi^c=\f$ 
(where symmetrically $\psi^c(x):= \inf \{ c(x,y)-\psi(x)\, :\, y\in Y\}$). Checking the latter condition amounts to verify that $\f$ enjoys
the so called {\em c-concavity} property $\f^{cc}=\f$ (in general we have only $\f^{cc}\ge \f$ whereas $\f^{ccc}=\f^c$).
We refer for instance to \cite{Vi} for further details about this c-duality.

Now, by exploiting Theorem \ref{bi} and Lemma \ref{polar}, we obtain a very simple proof of Kantorovich duality Theorem:
\begin{theorem}\label{dual0} \ \ The following
 duality formula holds:
$$  T_c(\mu,\nu) \ =\ \sup  \left\{  \int_X \f\, d\mu \ +\ \int_Y \psi\, d\nu\ :\ (\f,\psi) \in {\mathcal F}_c \right\} \ .$$
Moreover, the supremum in the second hand member is achieved by a pair $(\bar\f,\bar\psi)$ of conjugate c-concave functions
such that, for any optimal $\bar\gamma$ in (\ref{MK}) , there holds $\bar\f(x) +\bar\psi (y) = c(x,y)\ ,\ \bar\gamma$-a.e. .
\end{theorem} 
\begin{proof}\ By Theorem \ref{bi} and Lemma \ref{polar}, we have 
$$\begin{array}{lll} T_c(\mu,\nu)\ =\ F^{**}(\mu)\ &=& \ \sup \left\{ \int_X \f \, d\mu + \int_Y \f^c \, d\nu \ :\ \f\in C^0(X)\right\}\\
&\le& \sup  \left\{  \int_X \f\, d\mu \ +\ \int_Y \psi\, d\nu\ :\ (\f,\psi) \in {\mathcal F}_c \right\} \\
&\le& T_c(\mu,\nu) \ , \end{array}$$
where the last inequality follows from the definition of ${\mathcal F}_c$. Therefore $\inf (\ref{MK}) = \sup (\ref{MK*})$.
Furthermore, in the right hand side of first equality, we  increase the supremum by substituting $\f$ with $\f^{cc}$ (recall that $\f^{ccc}=\f^c$).
Thus $$\sup (\ref{MK*}) = \sup \left\{ \int_X \f \, d\mu + \int_Y \f^c \, d\nu \ :\  \f\in C^0(X) \ ,\ \f\  \text{c-concave} \right\}\ .$$
Take a maximizing sequence $(\f_n,\f_n^c)$ of  c-concave conjugate functions. 
It is easy to check that $\{f_n\}$ is equicontinuous on $X$: this
follows from the c-concavity property and from the uniform continuity of $c$ ( observe that  $ \f_n (x_1) - \f_n (x_2)= \f_n^{cc} (x_1) - \f_n^{cc} (x_2)
\le\sup_Y\{ c(x_1,\cdot)- c(x_2,\cdot)\}$). Then, by Ascoli's Theorem,  possibly passing to subsequences, we may assume that:
$ \f_n - c_n $ converges uniformly to some continuous function $\bar\f$ where $\{c_n\}$ is a suitable sequence of reals.
Then, one checks that $\bar\f$ is still c-concave and that $(\f_n - c_n)^c= \f_n^c+c_n$ converges uniformly to $\bar\f^c$. Thus, recalling that
 $\mu(X)=\nu(Y)=1$, we deduce that:
$$ \begin{array}{lll} \sup (\ref{MK*}) &=& \lim_n \left(\int_X \f_n \, d\mu + \int_Y \f_n^c \, d\nu\right ) \\
&=& \lim_n \left[\int_X (\f_n\!-\!c_n) \, d\mu + \int_Y (\f_n^c\!+\!c_n) \, d\nu\right]\\
 &=& \int_X \bar\f \, d\mu + \int_Y {\bar\f}^c \,d\nu \ .\end{array}$$
The last assertion is a consequence of the extremality relation:
$$ 0 \ =\ \inf (\ref{MK}) - \sup (\ref{MK*})\ =\ \int_{X\times Y} \left( c(x,y) - \bar\f(x) -\bar\psi(y)\right) \, \bar\gamma(dxdy)\ .$$

\end{proof}

\begin{rmk}\label{p=2}\rm \ a) \ In their discrete version (i.e. $\mu, \nu$ are atomic measures), problems (\ref{MK}) and (\ref{MK*}) can be seen as
particular linear programming problems (see section 3.1) 

\med b)\ The case $X=Y \subset \R^n$ and $c(x,y)= \frac{1}{2} |x-y|^2$ is important. In this case, the notion of c-concavity is linked to convexity and 
Fenchel transform since, for every $\f\in C^0(X)$, one has \  $ \frac{|\cdot|^2}{2} - \f^c =(\frac{|\cdot|^2}{2} - \f)^*$.
Then if $(\bar \f, \bar\f^c)$ is a solution of (\ref{MK*}), we find that $\f_0(x):=  \frac{|x|^2}{2} - \bar\f(x)$ is convex continuous and that the 
extremality condition: $\bar\f(x) +\bar\f^c (y) = c(x,y)$ is equivalent to Fenchel equality $\f_0(x) + \f_0^*(y)= (x| y)$.
Therefore, any optimal $\bar \gamma$ is supported in the graph of the subdifferential map $\partial\f_0$. 
In the case where $\mu$ is absolutely continuous with respect to the Lebesgue measure, it is then easy to deduce that the optimal $\bar \gamma$ is unique
and that $\bar \gamma=\gamma_{T_0}$ where $T_0=\nabla \f_0$ is the unique gradient (a.e. defined ) of a convex function such that $\nabla
\f_0^\sharp(\mu)=\nu$. This is a celebrated result by Y. Brenier (see for instance in the monographs \cite{Ev,Vi})
\end{rmk}

\subsection{The distance case} In the following, we assume that $X=Y$ and that $c(x,y)$ is a semi-distance.
As an immediate consequence of the triangular inequality, we have the following equivalence:
$$ \f \quad \text{c-concave} \quad \Leftrightarrow \quad \f(x)-\f(y) \le c(x,y)\quad, \ \forall (x,y) \quad \Leftrightarrow \quad \f^c = -\f $$
Let us denote  $\ {\rm Lip}_1(X):=\left\{ u\in C^0(X) \ :\ u(x) -u(y)\le c(x,y) \right\}\ .$
The first assertion of Theorem \ref{dual0} becomes the {\em Kantorovich-Rubintein} duality formula:
\begin{equation}\label{dual1} T_c(\mu,\nu) \ =\ \max \left\{ \int_X u \, d(\mu-\nu) \ :\ u\in {\rm Lip}_1(X)\right\}\ .\end{equation}
As it appears, $T_c(\mu,\nu)$ depends only on the difference $f=\mu-\nu$ which belongs to the space $\M_0(X)$ of signed measure on $X$ with zero average.
Defining $N(f):=T_c(f^+,f^-)$ provides a semi-norm ({\em Kantorovich norm}) on $\M_0(X)$ (it turns out that $\M_0(X)$ is not complete and that in general
its completion is a strict subspace of the dual of ${\rm Lip}(X)$).

\med We will now specialize to the case where $X$ is a compact manifold equipped with a geodesic distance. This will allow us to link the original problem
to another  primal-dual formulation closer to that considered in section 3.2 and yielding to a connection with partial differential equations.
As a model example, let us assume that $K=\overline\O$ where $\O$ is a bounded connected open subset  of $\R^n$ with a Lipschitz boundary. Let
$\Sigma\subset \overline{\O}$ be  a compact subset (on which the transport will have zero cost) and define
\begin{equation}\label{geo} c(x,y):=\ \inf \left\{ {\mathcal H}^1(S\setminus \Sigma) \ :\ S \ \textrm{Lipschitz curve joining $x$ to $y$}\ ,\ S\subset
\overline{\O}
\right\}\ ,\end{equation}
where ${\mathcal H}^1$ denotes the one dimensional Hausdorff measure (length).
 It is easy to check that $$c(x,y)= \min \{ \delta_\O(x,y), \delta_\O(x,\Sigma)+ \delta_\O(y,\Sigma)\}\ ,$$ where 
$\delta_\O(x,y)$ is the geodesic distance on $\O$ (induced by the Euclidian norm) . Furthermore, the following characterization holds:
\begin{equation}\label{Lip} u\in {\rm Lip}_1(X) \ \Longleftrightarrow u\in W^{1,\infty}(\O) \ ,\ |\nabla u| \le 1 \ \textrm{a.e. in $\O$}\ ,\ u= cte\ 
\textrm{on}\
\Sigma\ .\end{equation}
Since $f:=\mu-\nu$ is balanced, the value of the constant on $\Sigma$ in (\ref{Lip}) is irrelevant and can be set to $0$. Thus we may rewrite
the right hand side member of (\ref{dual1})
in a equivalent way as 
\begin{equation}\label{dual2} \max \left\{ \int_{\overline\O} u \, df \ :\ u\in W^{1,\infty}(\O)\ ,\ |\nabla u|\le 1\ \textrm{
a.e. on $\O$}\ ,\
 u=0\ \textrm{  on $\Sigma$} \right\}.\end{equation}

We will now derive a new dual problem for (\ref{dual2}) by using Proposition \ref{combi}. To that aim, we consider
$X=C^1(\overline\O)$ (as a closed subspace of $W^{1,\infty}(\O)$), $Y= C^0({\overline\O};\R^n)$, $Y^*= \M_b({\overline\O};\R^n)$ and the operator
$A: u\in X\mapsto \nabla u\in Y$.
\begin{theorem}\label{measure} \ Let $\mu,\nu \in \P(\overline{\O})$, $f=\mu-\nu$ and $c$ defined by (\ref{geo}). Then:
\begin{equation}\label{dual3}
 T_c(\mu,\nu) \  =\ \min \left\{ \int_{\overline\O} |\lambda|\ :\ \lambda\in  \M_b({\overline\O};\R^n) \ ,\ 
-\div \lambda = f \ \textrm{  on }\ \overline{\O}\setminus\Sigma \right\}\end{equation}
where the divergence condition is intended in the  sense that $\int_{\overline\O} \lambda\cdot \nabla\f = \int_{\overline\O} \f \, df$, for all
$\f\in C^\infty$ compactly supported in $\R^n\setminus \Sigma$.
\end{theorem}
\begin{proof}\ (\textit{sketch})\ We apply Proposition\ref{combi} with 
$\phi(u)= - \int_{\overline\O} u \, df$ if $u=0$ on $\Sigma$  \ ($+\infty$\ otherwise) and
$\psi(v)= 0 \ $ if $|v|\le 1$ on $\overline\O$ \ ($+\infty$ \ otherwise). We obtain that the minimum $\alpha$ in (\ref{dual2}) is reached 
and that $\alpha =\beta$ where
 $$- \beta \ := \ \inf \left\{- \int_{\overline\O} u \, df \ :\ u\in C^1(\overline\O)\ ,\ |\nabla u|\le 1\ \textrm{ on }\ \O\,\  u=0\ \textrm{  on }\
\Sigma\right\}\ .$$
To prove that $\beta= T_c(\mu,\nu)$, we consider a maximizer $\bar u$ in (\ref{dual2}) and prove that it can be approximated uniformly by a sequence
$\{u_n\}$ of functions in $C^1(\overline{\O})$ which satisfy the same constraints. This technical part is done by truncation and convolution arguments
 (we refer to \cite{BoBuDe} for
details).

\end{proof}
\begin{rmk}\label{4.1}\ \rm By localizing the integral identity associated with (\ref{dual3}), it is possible to deduce the optimality conditions which
characterizes optimal pairs $(\bar u,\bar\lambda)$ for (\ref{dual2})(\ref{dual3}) (without requiring any regularity). This is done by using
a weak notion of
tangential gradient with respect to a measure (see
\cite{BoBuSe,BoFr1}).  If $\bar\lambda= \bar\sigma\, dx$ where $\sigma\in L^1(\O;\R^n)$ and if $\Sigma\subset \bd\O$,
then we find that
$\bar\sigma =a\, \nabla \bar u$ where the pair $(\bar u, a)$ solves the system:
$$\left\{\begin{array}{cclll} -\div ( a\, \nabla \bar u) &=& f & \textrm{on }\ \O &\quad \textit{(diffusion equation)} \\
|\nabla \bar u |  &=& 1 & \textrm{a.e. on}\ \{a>0\}  & \quad \textit{(eikonal equation)}\\
u&=& 0 & \textrm{a.e. on}\ \Sigma & \\
\disp\frac{\partial u}{\partial n} &=& 0 & \textrm{on }\ \Sigma &
\end{array}\right.$$

\end{rmk}

\begin{rmk}\label{4.2}\ \rm  Given a solution $\bar \gamma$ for (\ref{MK}), we can construct a solution $\bar\lambda$
for (\ref{dual3}) by selecting for every $(x,y)\in \spt(\bar\gamma)$ a geodesic curve $S_{xy}$ joining $x$ and $y$ (possibly passing trough
the free cost zone $\Sigma$) and by setting, for every test $\phi$:
$$ \langle \bar\lambda, \phi\rangle := \int_{\overline{\O}\times \overline{\O}} \left(\int_{S_{xy}} \phi\cdot \tau_{S_{xy}} d {\mathcal H}^1 \right)
 \bar\lambda(dxdy)\ ,$$
where $\tau_{S_{xy}}$ denote the unit oriented  tangent vector (see \cite{BoBu}). It is also possible to show (see \cite{Am}) that any solution
$\bar\lambda$ can be represented as before through a particular solution $\bar \gamma$. As a consequence, the support of any solution
$\bar\gamma$ of (\ref{dual3}) is supported in the geodecic envelope of the set $\spt(\mu) \cup\spt(\nu) \cup \Sigma\ . $ However we
 stress the fact that, in general, there
is no uniqueness at all  of the optimal triple
$(\bar\gamma,\bar u,\bar\lambda)$ for (\ref{MK})(\ref{dual2})(\ref{dual3}).

\end{rmk} 
\begin{rmk}\label{4.3}\ \rm An approximation procedure for particular solutions of problems (\ref{dual2})(\ref{dual3})
 can be obtained by solving a p-Laplace equation and then by sending $p$ to infinity. Precisely,
consider the solution $u_p\in W^{1,p}(\O)$ of 
$$ -\div (|\nabla u|^{p-2} \nabla u)\ =\ f \quad \textrm{on $\overline{\O}\setminus\Sigma$} \quad ,\quad u=0 \ \textrm{on $\Sigma$} \ ,$$ 
which, for $p>n$, exists (due to the compact embedding $W^{1,p}(\O)\subset C^0(\overline{\O}$)) and is unique. 
In \cite{BoBuDe} it is proved that the sequence $\{(u_p,\sigma_p)\}$, where $\sigma_p= |\nabla u_p|^{p-2} \nabla u_p$, is relatively compact
in $\M_b(\overline{\O};\R^n) \times C^0(\overline{\O}$ (weakly star with respect to the first component) and that every cluster point
 $(\bar u,\bar \lambda)$ solves (\ref{dual2})(\ref{dual3}). It is an open problem to know if wether or not such a cluster point is unique.
If the answer is yes, the process decribed above would  select one optimal pair among all possible solutions. As far as problem (\ref{dual2}) is concerned,
this problem is connected with the theory of {\em viscosity solutions} for the {\em infinite Laplacian} (see \cite{Ev}) although this theory does not
provide an answer as it erases the role of the source term $f$.  On the other
hand, a new {\em entropy selection principle} should be found for the solutions of dual problem (\ref{dual3}). In fact, the following partial result holds:
Let $E: \M_b(\overline{\O};\R^n) \to \R\cup\{+\infty\}$ be the functional defined by
$$ E(\lambda) := \int_\O |\sigma| \log (|\sigma|) \,dx \quad \textrm{if $\lambda \ll dx$ and $\sigma=\frac{d\lambda}{d |\lambda|}$} \quad,\quad
+\infty \quad \textrm{otherwise}\ .$$
Assume that (\ref{dual3}) admits at least one solution  $\lambda_0$ such taht $E(\lambda_0) <+\infty$. Then it can be shown that the sequence
$\{\sigma_p\}$ does converge weakly-star to
$\bar\lambda$  the unique minimizer of the problem
$$ \inf \left\{ E(\lambda) \ :\ \lambda \ \textrm{solution of (\ref{dual3})}\right\}\ .$$
The general case, in particular when  all optimal measures are singular, is open.

\end{rmk}
\begin{rmk}\label{4.4}\ \rm Variational problems  (\ref{dual2})(\ref{dual3}) have important counterparts in the theory of elasticity and in
 optimal design problems (see \cite{BoBu}). They read respectively as
$$ \begin{array}{ll} &\disp \max \left\{ \int_{\overline\O} u \cdot df \ :\ u\in \cap_{p>1} W^{1,p}(\O;\R^n)\ ,\ \nabla u(x) \in K\  \textrm{a.e. on}\ \O\
,\
 u=0\ \textrm{  on }\ \Sigma \right\}\ ,\\
&\disp \min \left\{ \int_{\overline\O} \rho^0_K(\lambda)\ :\ \lambda\in  \M_b({\overline\O};\R^{n^2}_{\rm sym}) \ ,\ 
-\div \lambda = f \ \textrm{  on }\ \overline{\O}\setminus\Sigma \right\}\ ,\end{array}$$
\end{rmk}
where $K\subset \R^{n^2}_{\rm sym})$ is a convex  compact subset of symmetric second order tensors associated with the elastic material, \ $\rho^0_K(\xi)=
\sup\{\xi\cdot z : z\in K\}$ is convex positively one homogeneous and the functional on measures $\int_{\overline\O} \rho^0_K(\lambda)$ is intended in
the sense given in (\ref{H}). A celebrated example is given by {\em Michell's problem} (\cite{Mi}) where $n=2$ and $K:=\{z\in \R^{n^2}_{\rm sym},
|\rho(z)|\le 1\}$, being $\rho(z)$ the largest singular value of $z$. The potential $\rho^0_K$ is given by the non differentiable convex function 
$\rho^0_K(\xi)= \tau_1(\xi) + \tau_2(\xi)$, where the $\tau_i(\xi)$'s are the singular values of $\xi$.

Unfortunately it is  not known if the  vector variational problems above can be linked to an optimal transportation problem of the kind (\ref{MK}), even
if  the analoguous of equivalence (\ref{Lip}) does exist in the  Michell's case, namely (for $\O$ convex):
$$ \rho(e(u)) \le 1 \quad \textrm{on $\O$ \ }\quad  \Longleftrightarrow\quad  |(u(x)-u(y)|x-y)| \le |x-y|^2 \ ,\ \forall (x,y)\ .$$

\end{document}